\documentclass[a4paper,11pt]{article}
\usepackage{latexsym}
\usepackage{amsmath}
\usepackage{amsthm}
\usepackage{amssymb}
\usepackage{amsfonts}
\usepackage{esint}
\usepackage[all]{xy}
\usepackage{graphics,graphicx}
\usepackage{accents}
\usepackage[utf8]{inputenc}
\usepackage[T1]{fontenc}
\usepackage{color}

\theoremstyle{definition}
\newtheorem{defin}{Definition}[section]
\theoremstyle{plain}
\newtheorem{theo}[defin]{Theorem}
\newtheorem{prop}[defin]{Proposition}
\newtheorem{cor}[defin]{Corollary}
\newtheorem{lem}[defin]{Lemma}
\theoremstyle{definition}
\newtheorem{rem}[defin]{Remark}

\def\Q{{\mathbf Q}}
\def\Z{{\mathbf Z}}

\def\Im{ \operatorname{Im} }

\def \End {\operatorname{End} }
\def \Aut {\operatorname{Aut} }

\author{Séverin Philip}

\begin{document}

\title{Fields of definition of abelian subvarieties}
\maketitle
\begin{abstract}
In this paper we study the field of definition of abelian subvarieties $B\subset A_{\overline{K}}$ for an abelian variety $A$ over a field $K$ of characteristic $0$. We show that, provided that no isotypic component of $A_{\overline{K}}$ is simple, there are infinitely many abelian subvarieties of $A_{\overline{K}}$ with field of definition $K_A$, the field of definition of the endomorphisms of $A_{\overline{K}}$. This result combined with earlier work of Rémond gives an explicit maximum for the minimal degree of a field extension over which an abelian subvariety of $A_{\overline{K}}$ is defined with varying $A$ of fixed dimension and $K$ of characteristic $0$.
\end{abstract}

\section{Introduction}

Let $A$ be an abelian variety over a field $K$ of characteristic zero. If $B$ is an abelian subvariety of $A_{\overline{K}}$ we recall that there exists a smallest field $L$ such that $B$ is defined as a subvariety of $A_L$ over $L$. By this, we mean that there is a subvariety $C$ of $A_L$ such that $C_{\overline{K}} =B$. We will call this field $L,$ which is finite over $K$, the field of definition of $B$. We are interested in this paper in the link between the field of definition of the endomorphisms of $A_{\overline{K}}$ that we will note $K_A$ (as in \cite{r}) and $L$. The first remark is that every abelian subvariety $B$ of $A_{\overline{K}}$ is the image of some endomorphism $\varphi$ and thus we have an inclusion $L\subset K_A$. Our goal is to show that equality can and usually does happen. We prove the following theorem.

\begin{theo} \label{maintheo}
Let $A$ be an abelian variety over a field $K$ of characteristic zero, $r\geq 1$ an integer and integers $n_1,\dots, n_r$ greater or equal to $2$ such that $A_{\overline{K}}$ is isogenous to a product $\prod\limits_{i=1}^r C_i^{n_i}$ for simple pairwise non isogenous abelian varieties $C_i$ over $\overline{K}$. For all $r$-uples of integers $(k_1,\dots, k_r)$ with $k_i\in \{1,\dots, n_i-1\}$ for each $i\in \{1,\dots ,r\}$ there are infinitely many abelian subvarieties $B$ of $A_{\overline{K}}$ isogenous to $\prod\limits_{i=1}^r C_i^{k_i}$ with field of definition $K_A$. 
\end{theo}

 The main result of \cite{r} gives a bound on the degree $[K_A:K]$ as a function of the dimension $g$ of $A$, namely we have
 $$[K_A:K]\leq f(g)$$
where $f(g)=2\alpha(g)6^{g-1}g!$ with $\alpha(g)=1$ if $g\notin \{2,4,6\}$ and $\alpha(2)=2$, $\alpha(4)=5$, $\alpha(6)=7/6$.

 The theorem of \cite{r} also states that the bound for $[K_A:K]$ is achieved for some abelian varieties over $\Q$ isomorphic over $\overline{\Q}$ to a power of an elliptic curve. The application of Theorem \ref{maintheo} to this setting gives an explicit value for the supremum of the degrees $[L:K]$ obtained as above in terms of a function of $g$. More precisely we have the corollary.
\begin{cor}
Let $g\geq 2$ be a positive integer and $A$ an abelian variety of dimension $g$ over a field $K$ of characteristic zero. Then for any abelian subvariety $B\subset A_{\overline{K}}$ with field of definition $L$ we have
$$[L:K]\leq f(g).$$
Moreover, there are an abelian variety $A$ of dimension $g$ over $\Q$ and an abelian subvariety $B\subset A_{\overline{\Q}}$ with field of definition $L$ such that
$$[L:\Q]= f(g).$$ 
\end{cor} 

This result complements Proposition 4.2 of \cite{gb2019}, stating that the bound is sharp.

The first part of the paper is devoted to some preliminaries on linear algebra over skew fields. Specifically we study the automorphisms of $\mathrm{M}_n (D)$ for a skew field $D$ and their natural action on the right vector subspaces of $D^n.$ In order to do so we will show that such an automorphism is always of the form $f\sigma$ where $f$ is the conjugation by an element of $\mathrm{GL}_n(D)$ and $\sigma$ is an automorphism of $D$. 

We will use this to define an action on the grassmannian variety $\mathrm{Gr}_{k,n}(D)$ for an automorphism of $\mathrm{M}_n(D)$ that is compatible with the natural action on the right ideals. 

Finally, the Galois action of $\mathrm{Gal}(\overline{K}/K)$ on subvarieties factors through a finite group $G$ and translates into an action on the ideals of $\End A_{\overline{K}}\otimes \Q$ which from the previous construction yields an action on the grassmannian variety. Using the well-known link between ideals of $\End A_{\overline{K}}\otimes \Q$ and the abelian subvarieties of $A_{\overline{K}}$ the result comes down to the fact that the fixed points of a non trivial element of $G$ are given by the rational points of a closed subvariety properly contained in the grassmannian which is irreducible. The finite union of those sets cannot cover it and by density of the rational points we find many subvarieties that fit our needs.

\section{Automorphisms of simple algebras}

We start by recalling a well-known result on right ideals of a matrix algebra $\mathrm{M}_n(D)$ over a skew field $D$ which is Proposition 13.1 of \cite{saltdiv}. For the rest of this section we fix a skew field $D$ of finite dimension over $\Q$ and a positive integer $n$. Let $e=\{e_1,\dots, e_n\}$ be the canonical basis of $D^n$. We identify matrices in this basis and linear maps.

\begin{prop}\label{corres}
There is a bijective correspondence between right ideals of $\mathrm{M}_n(D)$ and right vector subspaces of $D^n$.
\end{prop}
 The correspondence is given in the following way. Let $V$ be a right vector subspace of $D^n$. We map it to the right ideal $I_V$ given by the set of linear maps with image in $V$. Given a right ideal $I$ we have a vector subspace $V_I$ given by the sum of the images of the elements of $I$. Moreover one can prove that there is an idempotent element $\phi\in I$ such that $V_I=\Im \phi$ and $I$ is generated by $\phi$ as a right ideal. From this we deduce the equality
 $$I=\bigoplus\limits_{i=1}^n V_I$$
 which states that $I$ is the set of matrices with elements of $V_I$ as vector columns. \\
 
 Let $\sigma\in \operatorname{Aut}_{\Q}D$. We can naturally extend $\sigma$ to an automorphism of $\mathrm{M}_n(D)$ by letting it act on the entries of the matrices. With the previous description of right ideals of $\mathrm{M}_n(D)$ we get that if $I=\bigoplus\limits_{i=1}^n V_I$ is a right ideal, then
 $$\sigma(I) =\bigoplus\limits_{i=1}^n \sigma(V_I)$$
 so that $\sigma(I)$ is mapped through our bijection to $\sigma(V_I)$.  \\
 
 We now turn to $\operatorname{Aut}_{\Q} \mathrm{M}_n(D)$ to understand its action on the right subspaces of $D^n$. Let $F$ be the centre of $D$, it is a finite extension of $\Q$. We have exact sequences 
$$\xymatrix{1 \ar[r] & \mathrm{Aut}_F \mathrm{M}_n(D) \ar[r] & \mathrm{Aut}_{\Q} \mathrm{M}_n(D) \ar[r] & \mathrm{Aut}_{\Q} F
} ;$$
$$\xymatrix{ 1 \ar[r] & \mathrm{Aut}_F D \ar[r] & \mathrm{Aut}_{\Q}D \ar[r] & \mathrm{Aut}_{\Q}F}.$$
The last arrow on the right is the restriction map. We also have the extension map that makes the following triangle commute
$$\xymatrix{
\mathrm{Aut}_{\Q} \mathrm{M}_n(D) \ar[r] & \mathrm{Aut}_{\Q} F\\
 \mathrm{Aut}_{\Q} D \ar[u]^{Ext} \ar[ur] & \\
}.$$ 
We will show that we get in this way a good set of preimages of the restriction morphism $\operatorname{Aut}_{\Q} \mathrm{M}_n(D)\rightarrow \operatorname{Aut}_{\Q} F$. We start with a lemma.
\begin{lem}
Let $F$ be a field, $\mathfrak{A}$ a simple $F$-algebra and $\sigma\in \operatorname{Aut}_{\Q} F$. There is an automorphism $\Phi$ of $\mathfrak{A}$ such that $\Phi_{|F} =\sigma$ if and only if
$$\mathfrak{A}\otimes_{F^{\sigma}} F \simeq \mathfrak{A}$$
where $F$ is seen as an $F$-algebra via $\sigma$.
\end{lem}
\begin{proof}
First let us assume $\mathfrak{A}\otimes_{F^{\sigma}} F\simeq  \mathfrak{A}.$ By definition of the tensor product there is an algebra homomorphism $\Phi$ such that the following diagram commutes
$$\xymatrix{
\mathfrak{A} \ar[r]^{\Phi} & \mathfrak{A} \\
F \ar[u] \ar[r]^{\sigma} & F \ar[u] 
}$$
the vertical arrows being the inclusions. As $\mathfrak{A}$ is simple and $\Phi$ nonzero it is an automorphism and the diagram gives the desired condition.

Now let $\Phi\in \operatorname{Aut}_{\Q} \mathfrak{A}$ such that $\Phi_{|F}=\sigma$. We show that $\mathfrak{A}$ is the tensor product $\mathfrak{A}\otimes_{F^{\sigma}} F$. Let $C$ be an $F$-algebra given by a map $i\colon F\rightarrow C$ and let $\alpha\colon \mathfrak{A} \rightarrow C$, $\beta \colon F\rightarrow C$ be maps such that the diagram
$$\xymatrix{
  & & C \\
  \mathfrak{A} \ar@/^1pc/[urr]^{\beta} \ar[r]^{\Phi} & \mathfrak{A}\ar@{-->}[ur] & \\
  F \ar[u] \ar[r]^{\sigma} & F\ar[u]\ar@/_1pc/[uur]_{\alpha} &
}$$
commutes. It is an easy check that $\beta\circ \Phi^{-1}$ is the dotted arrow and thus the proof is complete.
\end{proof}

Let $\sigma$ be an automorphism of $F$ such that there is a $\Phi\in \operatorname{Aut}_{\Q} \mathrm{M}_n(D)$ with $\Phi_{|F}=\sigma$. By the lemma we have an isomorphism 
$$\mathrm{M}_n(D) \simeq  \mathrm{M}_n(D)\otimes_{F^{\sigma}} F \simeq \mathrm{M}_n(D\otimes_{F^{\sigma}} F)$$
and so we get $D\simeq D\otimes_{F^{\sigma}} F$ by Theorem 1.9 of Chapter 8 of \cite{sch}. Using the lemma again we conclude that $\sigma$ lifts to an automorphism of $D$. 

From this remark we can now choose a set $\Sigma$ of representatives of the lifts of automorphisms of $F$ to $\mathrm{M}_n(D)$. Furthermore we can impose that for all $\sigma$ in the image of the map $\operatorname{Aut}_{\Q} \mathrm{M}_n(D) \rightarrow \operatorname{Aut}_{\Q} F$ there is $\tau \in \Sigma$ such that $\tau= \operatorname{Ext} \theta$ for some lift $\theta \in \operatorname{Aut}_{\Q} D$ of $\sigma$ to $D$.

We can now state the theorem we aimed at.

\begin{theo}\label{couple}
Let $f\in \operatorname{Aut}_{\Q} \mathrm{M}_n(D)$. There are a $P\in \mathrm{GL}_n(D)$ and a unique $\sigma\in \Sigma$ such that for all $M\in \mathrm{M}_n(D)$ we have
$$f(M)= P\sigma(M)P^{-1}.$$
\end{theo}
\begin{proof}
Let $f_{|F}=\sigma$, it is an automorphism of $F,$ and denote again $\sigma$ the lift of $\sigma$ to $\mathrm{M}_n(D)$ that is in $\Sigma$. As $f_{|F}=\sigma_{|F}$ we have $f\sigma^{-1} \in \operatorname{Aut}_F \mathrm{M}_n(D)$. By the theorem of Skolem-Noether there is a matrix $P\in \mathrm{GL}_n(D)$ such that $f\sigma^{-1}$ is given by conjugation by $P$. This gives the existence of the decomposition.

For the unicity, take $M=xI$ for $x\in F$. We have
$$f(M)=P\sigma(M)P^{-1} = \sigma(x)I$$
so that if $\tau \in \Sigma$ is another lift that works we get $\tau_{|F}=\sigma_{|F}$.
\end{proof}

In the rest of the text we will choose a pair $(P,\sigma)\in \mathrm{GL}_n(D)\times \Sigma$ given by the theorem for any $f\in \operatorname{Aut}_{\Q} \mathrm{M}_n(D)$. For such an $f$ we have a natural action on ideals of $\mathrm{M}_n(D)$ given by $I\mapsto f(I)$. By Lemma \ref{corres} this action induces an action on the vector subspaces of $D^n$ mapping a subspace $V$ to the subspace associated with the ideal $f(I_V)$. If $f$ is given by $(P,\sigma)$ we get $f(V)=P\sigma(V)$ where $P$ acts on the left.  

 We end this section by a useful lemma (that is classical in the context of vector spaces over fields) that tells us when this action is trivial.

\begin{lem}\label{classic}
Let $P\in \mathrm{GL}_n(D)$, $\sigma\in \Sigma$ and $k\in \{1,\dots, n-1\}$. If for all right vector subspaces $V$ of $D^n$ of dimension $k$ we have $P\sigma(V)=V$ then $P$ is a central homothety and $\sigma$ is the identity. 
\end{lem}
\begin{proof}
We argue by induction on $k$. If $k=1$ as $\sigma$ fixes the canonical basis $(e_i)_{i\in\{1,\dots,n\}}$ of $D^n$ we have $P(e_i)=e_i \lambda_i$ for some $\lambda_i \in D^{\times}$. For $1\leq j\leq n $ there exists $\mu_j \in D^{\times}$ such that
$$P\sigma(e_1+e_j)=(e_1+e_j)\mu_j=e_1\lambda_1 +e_j\lambda_j$$
by the hypothesis. It follows that $\mu_j=\lambda_1=\lambda_j$ and that $P$ is the left multiplication by $\lambda=\lambda_1$ so an homothety. Let $x\in D$ and consider the line $L=(e_1+e_2x)D$. We have $P\sigma(L)=L$ and $P\sigma(L)=\lambda \sigma(e_1+e_2x)D$. The vector $e_1+e_2x$ is on that line, which gives $\mu\in D^{\times}$ with
$$e_1+e_2x=(e_1\lambda+e_2\lambda \sigma(x))\mu = e_1 \lambda \mu + e_2\lambda \sigma(x) \mu.$$
Using again the fact that the $(e_i)$ form a basis we get $\mu=\lambda^{-1}$ and $\lambda \sigma(x) \mu =x$. Hence we get $\sigma_{|F}= \mathrm{id}$ so that $\sigma=\mathrm{id}$. This now gives $x=\lambda x\lambda^{-1}$ for all $x\in D$ which is the fact that $P$ is a central homothety.

We now assume $k>1$. Let $V,U$ be two right vector subspaces of $D^n$ of dimension $k$. We have $P\sigma(V)=V$ and $P\sigma(U)=U$ so
$$P\sigma(V)\cap P\sigma(U)=U\cap V$$
and as $P\sigma$ is a bijection 
$$U\cap V= P\sigma(U\cap V).$$
The result follows from the induction as any subspace of dimension $k-1$ is the intersection of two subspaces of dimension $k$.   
\end{proof}

The converse of the statement is obviously true.

\section{The grassmannian}

In this section we recall a construction of the grassmannian over a skew field. It is closely related to Severi-Brauer varieties as defined in \cite{saltdiv} chapter 13.

Let us consider the algebraic group $\mathrm{GL}_{dn}$ over $\Q$ where $d$ is the degree $[D:\Q]$. We will repeatedly use the following fact (compare with Propositions 4.4, 4.5 and 4.6 p. 246 of \cite{demgab}) :

{\em if $G$ is an algebraic group over a characteristic $0$ field $k$ and $S$ a subgroup of $G(k)$ that is closed for the induced topology, there exists a unique closed algebraic subgroup $H$ of $G$ such that $H(k)=S$ and $H(k)$ is dense in $H$. }

 Through the choice of a $\Q$-basis of $D,$ the group $\mathrm{GL}_n(D)$ is a closed subgroup of the $\Q$-points of $\mathrm{GL}_{dn}$ and by the fact there is a unique algebraic subgroup of $\mathrm{GL}_{dn},$ that we write $\mathrm{GL}_{n,D},$ such that its $\Q$-points are $\mathrm{GL}_n(D)$ and they are dense in $\mathrm{GL}_{n,D}$. 

There is a natural action of $\mathrm{GL}_n(D)$ on the right vector subspaces of $D^n$. Let $V_0$ be a right vector subspace of dimension $k$ of $D^n$. The stabilizer $\operatorname{Stab} V_0$ is a closed subgroup of $\mathrm{GL}_{n,D}(\Q)$ and the fact applies so that $\operatorname{Stab} V_0$ is the $\Q$-points of a unique closed subgroup of $\mathrm{GL}_{n,D}$ that we write $\operatorname{Stab} V_0$ again. We define the Grassmann variety $\mathrm{Gr}_{k,n}(D)$ to be the quotient $\mathrm{GL}_{n,D}/ \operatorname{Stab} V_0$. In particular it is an irreducible variety over $\Q$ and its $\Q$-points $\mathrm{Gr}_{k,n}(D)(\Q)=\mathrm{GL}_{n,D}(\Q)/ (\operatorname{Stab} V_0)(\Q)$ are dense in it. The last equality can be shown as in the proof of Proposition 13.2 of \cite{saltdiv}.

Let $I$ be an ideal of $\mathrm{M}_n(D),$ $f\in \operatorname{Aut}_{\Q} \mathrm{M}_n(D)$ associated to $(P,\sigma)$ and $V$ the right vector subspace corresponding to $I$, assume $V$ is of dimension $k$. We turn the action of $f$ on $V$ into an action of $f$ on $\mathrm{Gr}_{k,n}(D)$. Let $\varphi \in\mathrm{GL}_n(D)$ be such that $\varphi(V_0)=V$ and $\varphi_{\sigma}\in \mathrm{GL}_n(D)$ such that $\sigma(V_0)=\varphi_{\sigma}(V_0)$. We have that
$$f(V)=P\sigma(V)=P\sigma(\varphi)\varphi_{\sigma} V_0$$
as we can check explicitly by taking a basis of $V_0$.

It follows that we have a map
$$\begin{array}{cccccc}
a_f\colon & \mathrm{GL}_n(D) & \longrightarrow & \mathrm{GL}_n(D) &\longrightarrow & \mathrm{Gr}_{k,n}(D)(\Q) \\
& \varphi & \longmapsto &P\sigma(\varphi)\varphi_{\sigma} & \longmapsto & \overline{P\sigma(\varphi)\varphi_{\sigma}} 
\end{array}$$
This map comes from an algebraic $\Q$-morphism $a_f\colon \mathrm{GL}_{n,D}\rightarrow \mathrm{Gr}_{k,n}(D)$. We now check that it is constant on the equivalence classes so that it induces a map on the quotient. For $\Q$-points $\varphi$ and $\psi$ of $\mathrm{GL}_{n,D}$ in the same class we have by definition
$$\varphi(V_0)=\psi(V_0)$$
and so by construction $a_f(\varphi)=a_f(\psi)$. The result follows as the $\Q$-points are dense. From this we get a map $\overline{a_f}\colon \mathrm{Gr}_{k,n}(D)\rightarrow \mathrm{Gr}_{k,n}(D)$ such that if $I$ is an ideal of $\mathrm{M}_n(D)$ associated to $\overline{\varphi}$ then $f(I)$ is associated to $\overline{a_f}(\overline{\varphi})$.

We can thus define a closed subscheme $\operatorname{Fix}_f$ of $\mathrm{Gr}_{k,n}(D)$ as the equalizer of $\overline{\rm{id}}$ and $\overline{a_f}$ which represents the points of the grassmannian fixed by the action of $f$. By Lemma 2.4 this subscheme is properly contained in $\mathrm{Gr}_{k,n}(D)$ whenever $f\neq \rm{id}$. We have a description of the $\Q$-points of this subscheme as
$$\{v\in \mathrm{Gr}_{k,n}(D)(\Q) \mid f\cdot v=v\}.$$

\section{Existence of abelian subvarieties with $K_A$ as field of definition}

First we recall the correspondence between ideals of the algebra $\End A_{\overline{K}}\otimes \Q$ and abelian subvarieties of $A_{\overline{K}}$ for an abelian variety $A$ over a characteristic zero field $K$.  

\begin{prop}\label{bijssvar}
Let $A$ be an abelian variety over a field $K$.  There is a bijective correspondence between right ideals of $\End A_{\overline{K}} \otimes \Q$ and abelian subvarieties of $A_{\overline{K}}$ in the following way. An abelian subvariety $B$ of $A_{\overline{K}}$ is associated with the ideal 
$$I_B= \{\varphi \in \End A_{\overline{K}} \mid \Im \varphi \subset B\} \otimes \Q$$
and to a right ideal $I$ of $\End A_{\overline{K}} \otimes \Q$ we associate the abelian subvariety 
$$B_I= \sum\limits_{\varphi \in I\cap \End A_{\overline{K}}} \Im \varphi.$$
\end{prop}
\begin{proof}
This can for example be deduced from Lemma 2.2 of \cite{rvaom} with $C=A$.
\end{proof}

There are two important corollaries that will be used for our main result.

\begin{cor} \label{ssvardim2}
Let $A$ be an abelian variety over a characteristic zero field $K$ and $C$ a simple abelian variety over $\overline{K}$ with an isogeny 
$$\varphi \colon A_{\overline{K}} \longrightarrow C^n$$
for some positive integer $n$. 
If $D$ is the skew field $\End C\otimes \Q$ then $\varphi$ identifies $\End A_{\overline{K}} \otimes \Q$ and $\mathrm{M}_n (D)$. An abelian subvariety $B$ of $A_{\overline{K}}$ is isogenous to $C^{\ell}$ where $\ell$ is the dimension of the right vector subspace of $D^n$ associated to $I_B$ by Proposition \ref{corres}.
\end{cor}
\begin{proof}
We have that $I_B$ is mapped through Proposition \ref{corres} to 
$$V=\operatorname{Hom}(C,B)\otimes \Q.$$
From the fact that $C$ is a simple abelian variety we also have that $B$ is isogenous to $C^{\ell}$ for some $\ell$. We prove that $\ell= \dim_D V$. The isogeny between $B$ and $C^{\ell}$ gives isomorphisms of $\Q$-vector spaces
$$\operatorname{Hom} (C,B) \otimes \Q \simeq \operatorname{Hom}(C,C^{\ell})\otimes \Q\simeq \operatorname{End} (C)^{\ell} \otimes \Q \simeq D^{\ell}.$$
It follows that $V$ and $D^{\ell}$ have the same dimension over $\Q$ and thus the same dimension over $D$.
\end{proof}

\begin{cor}\label{galoisstable}
Let $A$ be an abelian variety over a characteristic zero field $K$ and $B$ an abelian subvariety of $A_{\overline{K}}$. Let $L$ be a finite extension of $K$. The right ideal $I_B$ of $\End A_{\overline{K}}\otimes \Q$ is stable for the action of $\mathrm{Gal}(\overline{K}/L)$ if and only if so is $B$.
\end{cor}
\begin{proof}
We show that the bijection from Theorem \ref{bijssvar} is compatible with the action of the Galois group $\mathrm{Gal}(\overline{K}/K)$. 

For $\sigma\in \mathrm{Gal}(\overline{K}/K)$ and $\varphi\in \End A_{\overline{K}}$ we have
$$\sigma(\Im \varphi) =\Im \sigma(\varphi)$$
since
$$\sigma(\Im \varphi) =\sigma (\{\varphi(x) \mid x\in A(\overline{K}) \}) =\{ \sigma(\varphi) (\sigma(x)) \mid x\in A(\overline{K}) \}= \Im \sigma(\varphi).$$
This gives directly $\sigma(B_I)=B_{\sigma(I)}$, hence the statement.
\end{proof}

We now start working towards the proof of our main theorem. The geometrical part of the proof is contained in the following statement.

\begin{theo}\label{existence}
Let $D$ be a skew field of finite dimension over $\Q$ and $G$ a finite subgroup of $\operatorname{Aut}_{\Q} \mathrm{M}_n(D)$. Then for all $k\in \{ 1,\dots, n-1\}$ there are infinitely many right ideals $I$ of $\mathrm{M}_n(D)$ associated to vector subspaces $V$ of dimension $k$ which are stable by no element of $G\setminus\{1\}$.  

\end{theo}
\begin{proof}
By Theorem \ref{couple} we may write elements of $G$ as couples $(P,\sigma)$ where $P\in \mathrm{GL}_n(D)$ and $\sigma\in \Sigma$. A right ideal $I$ is stable by an element $g\in G$ if and only if $g\cdot I= P\sigma(I)=I$ hence if and only if $P\sigma(V_I)=V_I$. Consider the set
$$S=\{ v\in \mathrm{Gr}_{k,n}(D)(\Q) \mid \exists g\in G\setminus \{1\},~ g\cdot v=v\}$$
of vector subspaces of dimension $k$ that are stable by at least one non trivial element of $G.$ We have
\begin{align*}S&=\bigcup \limits_{g\in G\setminus\{1\}} \{v\in \mathrm{Gr}_{k,n}(D)(\Q) \mid g\cdot v=v\}\\
&= \bigcup \limits_{g\in G\setminus\{1\}} \mathrm{Fix}_g (\Q).
\end{align*}

By Section 3, $\mathrm{Gr}_{k,n}(D)\setminus \bigcup \limits_{g\in G\setminus\{1\}} \mathrm{Fix}_g$ is a non-empty open subscheme by irreducibility and thus contains infinitely many $\Q$-points by density. These points yield vector subspaces of $D^n$ which in turn yield ideals with the desired property. 
\end{proof}

The last ingredient is an easy combinatorial lemma.

\begin{lem} \label{combin}
Let $S_1,\dots, S_r$ be sets and $H_i\subset \Aut S_i$ subgroups such that for any finite subgroup $G_i\subset H_i$ the set
$$\{ s\in S_i \mid \forall \lambda \in G_i\setminus \{\mathrm{id}\},~ \lambda s\neq s\}$$
is infinite. Then if $G\subset \Aut (\prod\limits_{i=1}^r S_i)$ is a finite subgroup such that for all $g\in G$ there exist $\tau_g \in \mathfrak{S}_r$ and bijections $\lambda_{g,i}\colon S_{\tau_g^{-1}(i)}\rightarrow S_{i}$ with
$$\begin{cases}
g(s_1,\dots, s_r)=(\lambda_{g,1}(s_{\tau_g^{-1}(1)}), \dots, \lambda_{g,r}(s_{\tau_g^{-1}(r)})) \\
\text{if } \tau_g(i)=i \text{ then } \lambda_{g,i} \in H_i
\end{cases}$$
the set 
$$F=\{ s\in \prod\limits_{i=1}^r S_i \mid \forall g\in G\setminus\{\mathrm{id}\}, ~ gs\neq s\}$$
is infinite.
\end{lem}
\begin{proof}
Set $G_i=\{\lambda\in H_i \mid \exists g\in G,~\tau_g(i)=i, \lambda_{g,i}=\lambda\}.$ This is a finite subgroup of $H_i$ as $G$ is finite. We show that we can build elements of $F$ by induction with infinitely many choices at each step. Assume we have chosen the first $i-1$ coordinates, $s_1,\dots, s_{i-1}$, of a candidate. Then we need to choose $s_i$ such that
\begin{itemize}
\item[(i)] $\forall \lambda\in G_i\setminus \{\mathrm{id}\},~\lambda(s_i)\neq s_i$
\item[(ii)] $\forall j,~1\leq j\leq i-1, ~\forall g\in G$ with $\tau_g^{-1}(i)=j$, $s_i \neq \lambda_{g,i}(s_j).$ 
\end{itemize}
As the second condition removes only finitely many choices ($G$ being finite) and there are infinitely many choices satisfying (i) from the hypothesis on the $H_i,$ we can choose infinitely many $s_i$ that fit our need and conclude by induction on $i$. 
\end{proof}

We are now able to prove Theorem \ref{maintheo}.

\begin{proof}
Let $k_1,\dots,k_r$ be integers satisfying the conditions of the statement.

The absolute Galois group $\mathrm{Gal}(\overline{K}/K)$ acts on the semi-simple algebra $\mathfrak{A}=\End A_{\overline{K}} \otimes \Q$ through its finite quotient $\mathrm{Gal}(K_A/K)$ which identifies with a finite subgroup $G$ of $\operatorname{Aut}_{\Q} \mathfrak{A}$. By Corollary \ref{galoisstable} we are led to look for ideals $I$ of $\mathfrak{A}$ which are stable by no element of $G\setminus\{1\}$. Since $A_{\overline{K}}$ is isogenous to $\prod\limits_{i=1}^r C_i^{n_i}$ we have an isomorphism $\varphi \colon \mathfrak{A}\simeq \prod\limits_{i=1}^r \mathrm{M}_{n_i}(D_i)$ for some skew fields $D_i$. 

Now let $S_i=\{ \text{right ideals } I\subset \mathrm{M}_{n_i}(D_i) \mid  \dim_{D_i} V_I=k_i\}$ and $H_i\subset \Aut S_i$ the subgroup of those bijections induced by automorphisms of the algebra $\mathrm{M}_{n_i}(D_i)$. Theorem \ref{existence} states that these satisfy the conditions of Lemma \ref{combin}. Now we prove that $G$ satisfies the remaining conditions. Let $g\in G$ and let $e_1=(1,0,\dots,0),\dots, e_r=(0,\dots, 0,1)$ be the primitive central idempotents of $\prod\limits_{i=1}^r \mathrm{M}_{n_i}(D_i)$. Let $\tau_g\in \mathfrak{S}_r$ be the permutation induced by $g$ on the $e_i$. As $g$ is an algebra automorphism if $\tau_g(i)=j$ then $g$ induces an isomorphism $M_{n_i}(D_i)\rightarrow M_{n_j}(D_j)$ and so a bijection $S_i\rightarrow S_j$ which we set to be $\lambda_{g,j}$. From this setting we see that Lemma \ref{combin} applies and we get infinitely many ideals $I$ of $\prod\limits_{i=1}^r \mathrm{M}_{n_i}(D_i)$ stable by no element of $G\setminus\{1\}$ and such that 
$$I=I_1\times \cdots \times I_r$$
with $\dim_{D_i} V_{I_i}=k_i$.
Given one such ideal $I$ let $B$ be the abelian subvariety of $A_{\overline{K}}$ given by $\varphi^{-1}(I)$. By Corollary \ref{ssvardim2}, $B$ is isogenous to $\prod\limits_{i=1}^r C_i^{k_i}$ and Corollary \ref{galoisstable} gives that the field of definition of $B$ is $K_A$ by construction. 
\end{proof}

The corollary announced in the introduction is a direct consequence of this result applied to the abelian varieties of Theorem 1.1 of \cite{r}.

\begin{rem}
Let $E$ be the elliptic curve over $\Q$ given by the equation $y^2=x^3-x$ and $C$ the one given by the equation $y^2=x^3+3x-2$. One can check the following equalities
\begin{align*}
\Q_E&=\Q(i),\\
\Q_C&=\Q, \\
\End E_{\overline{\Q}}&=\Z[i],\\
\End C_{\overline{\Q}}&=\Z.
\end{align*}
Consider now the abelian variety $A=E\times C^2.$ We have $\Q_A=\Q(i)$ but there is no abelian subvariety of $A_{\overline{\Q}}$ for which the field of definition is $\Q(i)$. This example proves that one cannot allow one of the indices $n_i$ to be equal to $1$ in the statement of Theorem \ref{maintheo}. Moreover by looking at $A'=E^2\times C^2$ we get that one cannot expect to choose one of the indices $k_i$ to be equal to $n_i$. 
\end{rem}

\end{document}